\documentclass[10pt]{article}

\usepackage{amsmath}
 \usepackage{amscd}
 \usepackage{amssymb}
 \usepackage{theorem}

\usepackage[toc,page]{appendix}

\usepackage{color}

 \usepackage{graphicx}
\usepackage{amsmath}
 \usepackage{amscd}
\usepackage{amssymb}
  \usepackage[all]{xy}

\newtheorem{TEO}{Theorem}[section]
\newtheorem{PROP}[TEO]{Proposition}
\newtheorem{LEM}[TEO]{Lemma}
\newtheorem{DEF}[TEO]{Definition}

\theorembodyfont{\normalfont}

\newtheorem{REM}[TEO]{Remark}

\theorembodyfont{\normalfont}

\newcommand\Oh{{\mathcal O}}

\newcommand\sB{{\mathcal B}}

\newcommand\sF{{\mathcal F}}
\newcommand\sG{{\mathcal G}}

\newcommand\sI{{\mathcal I}}

\newcommand\sL{{\mathcal L}}
\newcommand\sM{{\mathcal M}}

\newcommand\F{{\mathbb F}}
\newcommand\N{{\mathbb N}}
\newcommand\G{{\mathbb G}}

\newcommand\dual{\mathrel{\raise3pt\hbox{$\underline{\mathrm{\thinspace d
\thinspace}}$}}}

\newcommand\iso{\cong}

\newcommand\into{\hookrightarrow}

\newcommand\C{\mathbb C}
\newcommand\proj{\mathbb P}

\newcommand\Aut{\operatorname{Aut}}

\newcommand\im{\operatorname{Im}}

\newcommand\Cliff{\operatorname{Cliff}}
\newcommand\gon{\operatorname{gon}}

\newcommand\Hilb{\operatorname{Hilb}}

\newcommand\Pic{\operatorname{Pic}}

\newcommand\Sing{\operatorname{Sing}}

\newenvironment{proof}[1][]{\noindent\textbf{Proof#1}.  }{{\hfill $\blacksquare$}}

\begin{document}

\title{ Scrolls containing binary  curves}
\author{Marco Franciosi}
\date{}

\maketitle

\begin{abstract}
 We study families  of scrolls containing  a given rational curve and
families of rational curves contained in a fixed scroll
via  a  stratification  in terms of the degree of the induced map  
onto $\proj^1$ and we prove that  there is no rational normal scroll  of minimal degree and of dimension $\leq  \frac n2$ containing   a \textit{general} binary curve in $\proj^n$.

\hfill\break	
{\bf keyword:} algebraic curve,  rational normal scroll, binary curve

\hfill\break  {\bf Mathematics Subject Classification (2010)} 14H10, 14H51, 13D02, 14N25
\end{abstract}


\section{Introduction}

A classical result in birational geometry is 
 that a nondegenerate irreducible variety $X$ in $\proj^n$ has degree at least $\operatorname{codim}(X)+1$
and moreover if equality holds then either $X$ is a rational normal scroll, or it is a cone over the Veronese surface $\proj^2 \into \proj^5$, or it is a quadric hypersurface (see \cite[Theorem 1]{EH87}) .

The aim of this paper is the  study of  families  of scrolls containing  a fixed scheme (in particular a given rational curve) and
families of rational curves contained in a given  scroll  by considering the incidence correspondence of such families. 

We point out that by a rational normal scroll we mean not only the variety $\F(a_1, \ldots, a_d)$, but also its birational image in $\proj^n$, which may be singular 
 and that 
 by a family   of rational normal scrolls of dimension $d$ in $\proj^n$ we mean a subscheme of $\Hilb^{p}(\proj^n)$,  where 
 $p=p(h)$ is   the Hilbert polynomial of any irreducible normal scroll of dimension $d$ and degree $n-d+1$ in
$\proj^n$. 
 
 Bogomolov and Kulikov in \cite{BK} considered  the surface case,  by studying 
 the subschemes of $\Hilb(\proj^n)$  whose points correspond to the surface scrolls, resp.   to 
 cones over rational  normal  curves spanning $\proj^n$.
 
Here we firstly give an overview of families of rational normal scrolls satisfying some given conditions (see Thm. \ref{dimensionvarie}
which is a generalisation to every dimension of  \cite[Proposition 2]{BK}) 
and secondly    we study in detail the case in which it is given a rational normal curve $\Gamma$ passing through a set $S$ 
of $n+2$ points  in linear general position in $\proj^n$. To be more precise
fixing $k\in \N$
we consider  the family $\sF(S,k)_{\underline{a}}$  of scrolls of dimension $d= \frac{n}{2}$  containing  
 $\Gamma$   such that the induced map $\Gamma \to \proj^1$ has degree $k$ (see Prop.  \ref{Fk}). 
We study the main properties of such family  and we  show that for  two distinct integers $h,k$  
such that $k+h > d $ then
$\dim (\sF(S,k)_{\underline{a}} \cap \sF(S,h)_{\underline{a}}) \leq 2n-3$   (see Prop.  \ref{2n-3}).

 On the other side,  
one of the characterisations of scrolls is that  
they  contain many rational curves. 
 Here we study the family of rational normal curves passing through   $n+2$  fixed  points
and contained in a fixed scroll scroll $\F$
of dimension $d$.  In Prop. \ref{number_rnc}
we show that  such family   has dimension at most $\# (S \cap \text{Sing} (\F))$, whereas 
 in 
Prop. \ref{unicity}  we show that if  such family  contains more than one single curve then it has strictly  positive dimension.


As an application of the above results,
 in section 5  we study   families of scrolls containing binary curves. 
A binary curve is a
 stable curve $C$  consisting of two
 irreducible  curves,  both isomorphic to $\proj^1$,  intersecting transversally in $ n+2 $
distinct points  (hence of arithmetic genus $p_a(C)=n+1$).  In the non-hyperelliptic case $|K_C|$
embeds  $C$ as the union of two rational normal curves  of degree $n$ in $\proj^n$ passing through
$n+2$ points in linear general position, so that we can identify $C$ with its image in $\proj^n$. We want to point out firstly, as shown  by Colombo and Frediani in  \cite{CoFre}, 
the usefulness of binary curves in studying moduli problem of curves of genus $p_a(C)$ and secondly 
  that  
the existence of a scroll containing the curve $C$  forces the non-vanishing of higher order 
 syzygies  of the  ideal $I_C$ (see e.g. \cite{aprodu-nagel}).  
 
 Theorem \ref{no_small_scroll}  shows that given  a general  binary curve
 $C= \Gamma_1 \cup \Gamma_2$  in $\proj^n$,  
then there is no rational normal scroll  of minimal degree and of dimension $\leq  \frac n2$ containing $C$.

\subsection*{Acknowledgements} 
The author is grateful for support by the PRIN project  2015EYPTSB$\_$010  ``Geometry of Algebraic Varieties'' of italian MIUR. 

The author would like to thank Elisa Tenni for deep  and stimulating discussions on these arguments.

\section{Notation and conventions }
We work over  the field $\C$. We will consider subschemes  of $\proj^n$, assuming  $n\geq 3$.

\hfill\break Throughout this paper by $S$ we will denote a set   consisting of $n+2$ points in linear general position in $\proj^n$, by $\Gamma$ we will denote a rational normal curve in $\proj^n$,
and by $\F$ a rational normal scroll  in $\proj^n$. 

\hfill\break To clarify our  point of view we recall the description of rational normal scroll as shown in \cite[Chapter 2]{miles}.

Let $\underline{a}$ be  a multi-index of $d$ non-negative integers $\underline{a}=\{a_1, \ldots, a_d\}$ such that $\sum a_i=n-d+1$. 
By $\F_{\underline{a}}= \F(a_1, \ldots, a_d)$  we mean  the scroll $\proj (\bigoplus_{i=1}^d \Oh_{\proj^1}(a_i))$. 
It has coordinates $(y_1, \ldots, y_d),(t_0, t_1)$ and the  map $\F(a_1, \ldots, a_d) \to \proj^1$ corresponds to the coordinates $(t_0, t_1)$. Moreover the divisor class group is a free group generated by $L$, the pullback of $\Oh_{\proj^1}(1)$, and $M$ a relative hyperplane, which is characterised by the property that the locus $(y_i=0)$ is a divisor in the class $-a_i L +M$  (equivalently 
$M$ is the class of any monomial $t_0^{b}t_1^{c} y_i$ with $b+c= a_i$).
If $a_i> 0$ for some $i$, then  $|M|$ is a base point free system which induces a map $\varphi_{|M|}\colon \F(a_1, \ldots, a_d) \to \proj^n$,  birational onto its image $\F$;  if moreover $a_i > 0$ for every $i$, $\varphi_{|M|}$ is an embedding.
Since no confusion should arise, we will call rational normal scroll not only the variety $\F(a_1, \ldots, a_d)$, but also its birational image $\F$  in $\proj^n$. 
On the contrary 
if $\F \subset \proj^n$ is a rational normal scroll  obtaining by joining points on $d$ distinct rational normal curves lying in complementary linear subspaces (see e.g.  \cite[Example 8.26]{book-harris}),  then  there exists a unique smooth rational normal scroll $\F(a_1, \ldots, a_d)$ with $a_1 \geq 0$ for every $i$, such that $\F$ 
is the birational image of $\F(a_1, \ldots, a_d)$,  $M = \Oh_{\proj^n}(1)_{| \F}$ and   $L$ the pullback of $\Oh_{\proj^1}(1)$  via the standard projection
$\pi:\F \dashrightarrow \proj^1$. 

\hfill\break
Let $\F \subset \proj^n$ be a rational normal scroll of dimension $d$ and  degree $n-d+1$. Let $p$ be its Hilbert polynomial. 
By a family   of rational normal scrolls of dimension $d$ and  degree $n-d+1$  in $\proj^n$ we mean a subscheme of $\Hilb^{p}(\proj^n)$.

\hfill\break
Throughout the paper e will constantly use the following notation for families of scroll:
\begin{itemize}
\item[$\sG$] the family of rational normal curves in $\proj^n$;
\item[$\sG(Y)$] the family of rational normal curves in $\sG$ containing a given scheme $Y \subset \proj^n$;
\item[$\sG^{\vee}(\F)$] the family of rational normal curves in $\sG$ contained in a given scroll $\F \subset \proj^n$;
\item[$\sG^{\vee}(\F,k)$] 
 the family of rational normal curves $\Gamma$ in $\F$ such that the
induced map $\Gamma \to \proj^1$ has degree $k$;
\item[$\sF$] the family of rational normal scrolls of dimension $d$ and degree $n-d+1$ in $\proj^n$;
\item[$\sF(Y)$] the family of rational normal scrolls in $\sF$ containing a given scheme $Y \subset \proj^n$;
\item[$\sF_{\underline{a}}$] the family of rational normal scrolls in $\F$ such that $\F \iso  \varphi_{|M|}( \F(a_1, \ldots, a_d))$;
\item[$\sF_0$] the family  of balanced scrolls, i.e. the family $\sF_{\underline{a}}$ corresponding to the multi-index $\underline{a}$ such that $|a_i - a_1| \leq 1$;  note that this $\underline{a}$ is unique up to permutation;
\item[$\sF(Y)_{\underline{a}}$] the family $\sF(Y) \cap \sF_{\underline{a}}$;
\item[$\Aut(\underline{a})$] the group $\Aut (\bigoplus_{i=1}^d\Oh_{\proj^1}(a_i))$.
\end{itemize}


\section{Families of rational normal scrolls}

In  this section
we point out a  theorem which collects some well known results on families of scrolls and of rational normal curves
 in $\proj^n$. 
We believe that
most of these  results  were known since the 19th century, but we were not able to find  an adequate reference.

\begin{TEO}\label{dimensionvarie}

\begin{enumerate}\renewcommand\labelenumi{(\alph{enumi})}

\item
The dimension of the family $\sF$ of rational normal scrolls of dimension $d$ and degree $n-d+1$ in $\proj^n$ is $n^2 +2n -2-d^2$.

\item  The dimension of the family $\sF_{\underline{a}}$ of rational normal scrolls in $\sF$ birational to the scroll $\F(\underline{a})= \F(a_1, \ldots, a_d)$  is $n^2 +2n -2-d^2- (\dim \Aut(\underline{a})-d^2)$.

\item
The family $\sF_0$ corresponding to scrolls $\sF_{\underline{a}}$  such that $|a_i - a_1| \leq 1 \ \forall i > 1$  is dense in the family $\sF$ of rational normal scrolls of dimension $d$ and degree $n-d+1$ in $\proj^n$.

\item  The dimension of the family  $\sG$ of rational normal curves in $\proj^n$  is  $n^2+2n-3$.

\item
Let $S$ be a  set of $n+2$ points in linear general  position in $\proj^n$ . Then the family $\sG(S)$ of rational normal curves containing $S$ has dimension $(n-1)$.

\item
Let $S$ be a set of $n+2$ points  in linear general  position  in $\proj^n$. Let $\sF(S)_{\underline{a}}$ be  the
 family of rational normal scrolls  in  $\sF_{\underline{a}}$ containing  $S$. Then
$$ \dim \sF(S)_{\underline{a}} =(n+2)\cdot d - (d^2+2)-(\dim \Aut(\underline{a})-d^2) $$

\end{enumerate}
\end{TEO}

\begin{proof}
Fix a suitable multi-index $\underline{a}$ and assume at first that $d \neq 2 $ or $d=2$ and  $a_1 \neq a_2$. The points of $\sF_{\underline{a}}$ parametrise the images of all possible morphisms $g: \F(\underline{a}) \to \proj^n$ given by the complete linear series $|M|$.
Moreover It is clear that all rational normal scrolls $X \in \F_{\underline{a}}$ are projectively equivalent in $\proj^n$.

Therefore fixing a morphism $g_0: \F(\underline{a}) \to \proj^n$, the family $\sF_{\underline{a}}$ is isomorphic to the quotient
$\operatorname{PGL}(n+1, \mathbb{K})/\Aut$, where the group $\Aut$ is an algebraic subgroup of $\operatorname{PGL}(n+1, \mathbb{K})$  fixing the scroll $g_0(\F(\underline{a}))$. In particular $\sF_{\underline{a}}$ is an irreducible quasi projective variety.

To complete the proof we must evaluate the dimension of $\Aut$. With this aim we will show that any automorphism of $\F(\underline{a})$ can be lifted to an element of $\operatorname{PGL}(n+1, \mathbb{K})$, i.e. $\Aut \iso \Aut(\F(\underline{a}))$, and we will estimate $\dim \Aut(\F(\underline{a}))$.

Since $\F(\underline{a}) \ncong \F(a,a)=\proj^1 \times \proj^1$ there is a unique projection map $\F(\underline{a}) \to \proj^1$, thus any automorphism on $\F(\underline{a})$ induces an automorphism on $\proj^1$. It is clear that the converse holds: any automorphism on $\proj^1$ can be lifted to an automorphism on $\F(\underline{a})$ and to an automorphism in $\Aut$.

Consider the following exact sequence
$$0 \to \Aut_0 \to \Aut(\F(\underline{a})) \to \Aut(\proj^1) \to 0 $$
where $\Aut_0$ is the subgroup of automorphisms of $\F(\underline{a})$ commuting with the projection to $\proj^1$. It is  immediately seen that $\Aut_0 \iso \proj (\Aut(\underline{a}))= \proj( \Aut (\bigoplus_{i=1}^d\Oh_{\proj^1}(a_i)))$ (see \cite[Proposition II.7.12]{Ha}) and  every automorphism can be lifted to an element of $\operatorname{PGL}(n+1, \mathbb{K})$. A simple induction argument shows that
\begin{equation}\label{d^2}\dim \Aut(\underline{a})\geq d^2\end{equation}
 and equality holds if and only if $|a_i-a_1|\leq 1$ for every $i$.

We conclude that \begin{eqnarray*}\dim \sF_{\underline{a}} &=& ((n+1)^2-1)- (\dim \Aut (\bigoplus_{i=1}^d\Oh_{\proj^1}(a_i))-1) -3\\
&=& n^2 +2n -2-d^2 -(\dim \Aut (\bigoplus_{i=1}^d\Oh_{\proj^1}(a_i))-d^2)\end{eqnarray*}
This exactly means that
$$\dim \sF_{\underline{a}} \leq n^2 +2n -2-d^2. $$
with equality holding if and only if $|a_i-a_1|\leq 1$ for every $i$.

It remains the case $\underline{a}=\{a,a\}$. In this case we need to take into account also  the automorphism of $\F(a,a) \iso \proj^1 \times \proj^1$ swapping the two $\proj^1$'s. See \cite{BK} for the details.

\hfill\break
The dimension of the family $\sG$ of rational normal curves is just a particular case of the case of rational normal scrolls.  See for example \cite{CoFre}.

\hfill\break
Finally, $\sF(S)_{\underline{a}}$ is isomorphic to the family $U / \Aut(\F(a_1, \ldots, a_d))$, where $U$ is the open subscheme of $\F(a_1, \ldots, a_d)^{n+2}$ given by $n+2$ points  imposing independent conditions on $|M|$.
Therefore
$$ \dim \sF(S)_{\underline{a}} =\dim \F(a_1, \ldots, a_d)^{n+2} - \dim  \Aut(\F(a_1, \ldots, a_d))$$
and we may conclude. 
\end{proof}

\begin{PROP}\label{H0_denso} The family $\sF_0$ of balanced scrolls is dense in the family $\sF$ of rational normal scrolls of dimension $d$ and degree $n-d+1$ in $\proj^n$.
\end{PROP}
\begin{proof} (See  \cite[chapter 2,Exercises 22, 23]{miles}).

Consider a multi-index $\underline{a}=(a_1, \ldots, a_d)$, and assume that $a_i \geq 0$ for every $i$ and $a_1 \geq 1$. We do not need to assume that $a_1 \leq a_2\leq \ldots \leq a_d$.
We will prove that there is a family of scrolls in $\proj^n$, each of them  image of $\F(a_1, \ldots, a_d)$,  which degenerates to a scroll which is image of $\F(a_1-1, a_2, \ldots, a_{d-1}, a_d+1)$. This is clearly enough to conclude the statement.

Consider at first the ($d+1$)-dimensional scroll $$\F^{d+1}=\F(a_1-1, a_2, \ldots, a_{d-1}, a_d,1) \to \proj^{n+1}$$
and let $M'$ be the relative hyperplane divisor.
 The scroll $\F(a_1, \ldots, a_d)$ belongs to the linear system $|M'|$ in $\F^{d+1}$
 since we have the embedding
$$\begin{matrix}
\varphi_1: & \F(a_1, \ldots, a_d)& \into & \F(a_1-1, a_2, \ldots, a_{d-1}, a_d,1)\\
& (x_1, \ldots, x_d)\times (t_0, t_1) & \mapsto& (t_0 x_1, x_2, \ldots, x_d, t_1^{a_1-1}x_1) \times (t_0, t_1)\\
\end{matrix}$$
Note that  $\im (\varphi_1)=\{t_0 y_{d+1}=t_1^{a_1-1}y_1\} \subset \F^{d+1}$.\\

Similarly we have a second embedding  

$$\begin{matrix}
\varphi_2: & \F(a_1-1, a_2, \ldots, a_{d-1}, a_d+1)& \into & \F(a_1-1, a_2, \ldots, a_{d-1}, a_d,1)\\
& (x_1, \ldots, x_d)\times (t_0, t_1) & \mapsto& (x_1, x_2, \ldots, t_0 x_d, t_1^{a_d}x_d) \times (t_0, t_1)
\end{matrix}$$
It is clear that $\im (\varphi_2)=\{t_0 y_{d+1}=t_1^{a_d}y_d\} \subset \F^{d+1}$.\\

Consider now the family of divisors in $|M'|$
given by the equation $$\{t_0 y_{d+1}= \lambda t_1^{a_1-1}y_1 +
t_1^{a_d}y_d\}  \mbox { for  } \lambda \in \C. $$
For $\lambda \neq 0$ we have a family of scrolls isomorphic to $\im (\varphi_1)$, which degenerates to  $\im (\varphi_2)$ when
$\lambda=0$.
\end{proof}

\section{Rational normal curves contained in  scrolls}



First of all let us recall the following characterisation  of rational normal curves contained in a scroll.
\begin{REM}\label{coordinate} Fix a multi-index $\underline{a}=\{a_1, \ldots, a_d\}$,  consider a  scroll  $\F(a_1, \ldots, a_d)$, and
the associated morphism: 
$\varphi: \F(a_1, \ldots, a_d) \to \proj^n$.

 Let $\F:= \varphi(\F(a_1, \ldots, a_d))$. By the birationality of $\varphi$,  a rational normal curve $\Gamma \subset \F$ corresponds to a unique embedding $\psi: \Gamma \into \F(a_1, \ldots, a_d)$ which can be explicitly written with respects to the coordinates $(s_0, s_1)$ on $\Gamma$:
\begin{center} $\begin{matrix}
\psi: & \Gamma & \to & \F(a_1, \ldots, a_d)\\
& (s_0, s_1) & \mapsto & (y_1(s_0,s_1), \ldots, y_d(s_0,s_1)) \times (t_0(s_0,s_1), t_1(s_0, s_1))
\end{matrix}$\end{center}
The induced map $\Gamma \to \proj^1$ is the composition:
\begin{center} $\begin{matrix}
\Gamma & \to & \proj^1\\
(s_0, s_1) & \mapsto & (t_0(s_0,s_1), t_1(s_0, s_1))
\end{matrix}$\end{center}
whose degree is $\deg t_0(s_0,s_1)= \deg t_1(s_0, s_1)=k$.

Since $\Gamma$ is a rational normal curve  in $\proj^n$ we have that $\deg (\Oh_{\proj^n}(1)_{|\Gamma})=\Gamma \cdot M=n$. The divisor $(y_i=0)\subset \F(a_1, \ldots, a_d)$ is in the linear system $|M-a_i L|$, hence $\deg y_i(s_0, s_1)=n-a_i k$.

\end{REM}

We are going  to study  family of rational normal curves  contained in a fixed scroll and conversely study rational normal scrolls
containing some  rational curves. 
 We start  with the following obvious remark.

\begin{REM}\label{proj-scroll} Let $\F$ be a rational normal scroll  of minimal degree in $\proj^n$ and let $P_0 \in \F$. Consider the projection map  from $P_0$
onto a general hyperplane 
$$\pi_{P_{0}} : \proj^n  \dashrightarrow \proj^{n-1}.$$
Then $\overline{\pi_{P_0}(\F)}$ is still a rational scroll of minimal degree in $\proj^{n-1}$. Moreover if $P_0$ is smooth then
$\dim (\overline{\pi_{P_0}(\F)}) = \dim (\F)$, whereas if $P_0$ is  singular then $\dim (\overline{\pi_{P_0}(\F)}) = \dim (\F) -1$.

\end{REM}

Our analysis is subdivided in  several steps.

\begin{PROP}\label{curve_in_scroll} Let $\F(a_1, \ldots, a_d)$ be a scroll of dimension $d$  and let $\F $ be its image  in $\proj^n$.

Let $\sG^{\vee}(\F,k)$ be 
 the family of rational normal curves $\Gamma$ in $\F$ such that the
induced map $\Gamma \to \proj^1$ has degree $k$. Then
\begin{itemize}
\item 
$\sG^{\vee}(\F,k)$  is irreducible and has dimension $(d-1)(n+3-k)+(k-1)(2d-n)$ if $k \leq d$ and $n-a_i k \geq 0$ for
every $i$ ;
\item 
$\sG^{\vee}(\F,k) =\emptyset$ otherwise.

\end{itemize} 
 \end{PROP}

 \begin{proof} Notice that
 there can not be a curve $\Gamma$ with $k > d$ since the general  fibre of  $\Gamma \to \proj^1$ consists of $k$ points
in linear general  position, contained in a fibre of $\F(a_1, \ldots, a_d)$, which is a $(d-1)-$dimensional projective
subspace of $\proj^n$. Therefore from now on we may assume that $k \leq d$.\\

 We have seen in Remark \ref{coordinate} that a rational normal curve $\Gamma \subset \F$ is determined by its coordinates
 \begin{center} $\begin{matrix}
\psi: & \Gamma & \to & \F(a_1, \ldots, a_d)\\
& (s_0, s_1) & \mapsto & (y_1(s_0,s_1), \ldots, y_d(s_0,s_1)) \times (t_0(s_0,s_1), t_1(s_0, s_1))
\end{matrix}$\end{center}
Such polynomials can not exist if $\deg y_i=n-k a_i <0$ for some $i$, hence in this case there are no rational curves
with the required properties. On the contrary if  $\deg y_i=n-k a_i \geq 0$ for every $i$ then the family is nonempty.

It is clear that the coordinates $\{y_1, \cdots,y_d, t_0, t_1\}$ are unique up to the action of $\Aut(\Gamma)$ and of $({\C^{\ast}})^2$ on $\F(a_1, \ldots,
a_d)$.
Therefore we get 
\begin{eqnarray*}\dim&=& \sum_{i=1}^d h^0(\proj^1, \Oh_{\proj^1}(n-k a_i)) + 2 h^0(\proj^1, \Oh_{\proj^1}(k))
-3-2\\
&=&(d-1)(n+3-k)+(k-1)(2d-n)
\end{eqnarray*}
This family is clearly irreducible, since it is dominated by an open set of a vector space.
\end{proof}

Now we analyse the simplest cases of rational normal scrolls:  hypersurface scrolls, i.e.,   quadrics of rank $3$ or $4$.
\begin{LEM} Let $Q= \{q(x_0, \ldots, x_n)=0\}$  be a quadric of rank 3 or 4 in $\proj^n$. Let $S \subset Q$ be a set of $n+2$ points in linear general  position in $\proj^n$ and assume that
$$ S \cap \text{Sing}(Q) = \emptyset.$$

Then there exists at most a finite number of rational normal curves $\Gamma$ such that $S \subset \Gamma \subset Q $.
\end{LEM}
\begin{proof} Assume that $S$ is the set of coordinate points $(0: \ldots: 0:1:0:\ldots, 0)$ in $\proj^n$ and  the point $(1: \ldots:1)$.\\

It is well known that, up to $\Aut(\proj^1)$, every rational normal curve $\Gamma$ through $S$ is given by the map
$$\begin{matrix}\varphi=\varphi_{\Gamma}:& \proj^1& \to& \proj^n\\
& (s_0, s_1) & \mapsto & (\frac{1}{s_0}:\frac{1}{s_0- s_1}:\frac{1}{s_0- a_2 s_1}:\ldots:\frac{1}{s_0-a_n s_1})
\end{matrix}$$
and that the numbers $a_i$ determine uniquely the curve. Since $S\subset Q$ we have that
\begin{equation}\label{parameters}
q(\varphi_{\Gamma}(s_0, s_1))= s_0 s_1 (s_0-s_1) \prod_{i=2}^n(s_0 - a_i s_1)\cdot p(s_0, s_1)
\end{equation}
where $p(s_0, s_1)$ is a degree $n-2$ polynomial in $(s_0, s_1)$ with coefficients depending on the parameters $a_i$.\\

The rational normal curves we are looking for correspond to  parameters $(a_2, \cdots, a_n) $ such that $ p(s_0, s_1) \equiv 0$ on $\proj^1$,
where $p$ is defined in  equation (\ref{parameters}) and it depends on $(a_2, \cdots a_n)$.
%
%
A simple Bertini's  type argument shows that for a general   choice of  the parameters $a_2, \cdots, a_n$ the polynomial $q(\varphi(s_0, s_1))$ has $2n$ distinct zeroes.
%
%
 Notice that the assumption $S \cap \text{Sing}(Q) = \emptyset$ is necessary since otherwise for every $\Gamma$ the polynomial
$q(\varphi_{\Gamma}(s_0,s_1))$ should vanish with order $\geq 2$ at  the preimage of every point in $S \cap \text{Sing}(Q)$.

Consider now  $n-1$ general points in in $\proj^1$: $T_1, \ldots, T_{n-1}$.   Let $\sG_0 =\sG(S)$ be  the family of rational normal curves in $\sG$ containing  $S$ and  for every $j=1,\cdots n-1$ define 
inductively $\sG_j$ to be the closure  in $\sG$ of  the family  of rational normal curves $\Gamma$ through $S$ such that $\varphi_{\Gamma}(T_1), \ldots, \varphi_{\Gamma}(T_j) \in Q \setminus S$. 
It is clear that $\sG_j$ corresponds to the family of parameters $\{a_2, \ldots, a_n\}$ such that
$$p(T_1)=\ldots =p(T_j)=0.$$

Notice that $\sG_{n-1}$   is the family  of rational normal curves contained in $Q$ since $\deg (p)= n-2$. 
Now it is $\dim \sG_0 =  n-1$  by Theorem  \ref{dimensionvarie}. Therefore 
the lemma follows if we show that
$$\dim \sG_{j+1}=\dim \sG_j - 1 . $$

Consider the incidence correspondence
$$I=\{(\Gamma, x): \varphi_{\Gamma}(x) \in Q \setminus S\} \subset \sG_j \times \proj^1 \setminus \{T_1, \ldots, T_j\}$$
The map $I \to \sG_j$ is generically finite thanks to the existence of generic map $\varphi$ we discuss above, thus $\dim I = \dim X_j$. Moreover the map $I \to \proj^1 \setminus \{T_1, \ldots, T_j\}$ is nonconstant by genericity  of  $\varphi$, thus the general   fibre has codimension 1.  Since the general   fibre
over $P_{j+1}$ is $\sG_{j+1}$  we  conclude.
\end{proof}

\begin{LEM}
 Let $Q= \{q(x_0, \ldots, x_n)=0\}$  be a quadric of rank 3 or 4 in $\proj^n$.  Let $S \subset Q$ be a set of $n+2$ points in  linear general  position in $\proj^n$. Then the family of rational normal curves $\Gamma$ such that $S \subset \Gamma \subset Q$ has dimension at most $\# (S \cap \text{Sing} (Q))$.

\end{LEM}
\begin{proof} It is a simple induction argument on $\# (S \cap \text{Sing} (Q))$. If $\# (S \cap \text{Sing} (Q))=0$ we are back to the previous Lemma. Assume that $\# (S \cap \text{Sing} (Q))\neq 0$.

Consider one point $P \in S \cap \text{Sing} (Q)$ and project away from it, $\pi_P: \proj^n \dashrightarrow \proj^{n-1}$. The degree 2 polynomial $q$ does not change via the projection since it is singular at $P$. Then the family of rational normal curves $\Gamma$ such that $S \subset \Gamma \subset Q$ is obviously given by all the possible liftings through $P$ of the rational normal curves $\Gamma \subset \proj^{n-1}$ such that $\pi_P(S\setminus P) \subset \Gamma \subset (q=0)$. Since every curve in $\proj^{n-1}$ has a one dimensional family of liftings, we conclude by induction.
\end{proof}

\hfill\break
Now we  give an estimate of the dimension of the family of rational normal curves passing through  $n+2$ general   points and contained in a  fixed rational normal scroll.
\begin{PROP}\label{number_rnc} Let $\F$ be a rational normal scroll of dimension $d$ and minimal degree in $\proj^n$, with $n \geq d+1$. Let $S \subset \F$ a set of $n+2$ points in linear general  position in $\proj^n$. Then the family  of rational normal curves $\Gamma$ such that $S \subset \Gamma \subset \F$ has dimension at most $\# (S \cap \text{Sing} (\F))$.

\end{PROP}
\begin{proof} We argue by induction on $n-d$. If $n-d=1$ then  $\F$ is a quadric of rank 3 or 4, and  we apply the previous Lemma.

Assume now  $n-d > 1.$

If $S \cap \text{Sing} (\F) \neq \emptyset$ we use the same projection argument adopted in the previous Lemma, and we conclude by induction.

Thus we may assume that $S \cap \text{Sing} (\F) = \emptyset$.

Take $P \in S$ and consider the projection $\pi_P: \proj^n \dashrightarrow \proj^{n-1}$.

 By Remark \ref{proj-scroll} the closure of $\pi_P (\F)$ is a rational normal scroll $\F_0$ of dimension $d$ and minimal degree in $\proj^{n-1}$.

Every rational normal curve $\Gamma$ in $\F$ through $S$ projects to  a rational normal curve $\Gamma_0$ in $\F_0$ passing through the image of the remaining $n+1$ points in $\proj^{n-1}$. Moreover  it is easy to check that none of these $n+1$ points is in the singular locus of $\F_0$, thus by induction we may assume that there is at most a finite number of rational normal curves through them.

A rational normal curve $\Gamma_0$ in the scroll $\F_0 \subset\proj^{n-1}$ has precisely one lifting in $\F$ by birationality, thus
it has at most one lifting which passes through $P$.
 Therefore, by induction, we may conclude.\end{proof}

\begin{LEM}\label{k1}
 Let  $\F=\F(a_1,\ldots, a_d)$ be a rational normal scroll,   $\pi: \F(a_1, \ldots, a_d)\to \proj^1$ be the usual projection, 
%
 and let
 $L = \pi^{\ast} (\Oh_{\proj^{1}}(1))$ and $M$  a  hyperplane divisor. 

Let  $S$ be a set of $n+2$ distinct points
 imposing independent conditions on the linear system $|M|$. Then  $\pi_{ |S}: S \to \proj^1$ is injective  if and only if there exists a smooth rational curve $\Gamma$ such that
\begin{enumerate}
\item $S \subset \Gamma \subset \F(a_1, \ldots, a_d)$;
\item $\Gamma \cdot L=1$;
\item $\Gamma\cdot M=n$.
\end{enumerate}
Moreover if such a  $\Gamma$ exists it is unique.

\end{LEM}

\begin{proof} 
Assume there exists a rational normal curve $\Gamma$ satisfying the conditions  {\em1., 2., 3.} Since $\Gamma \cdot L=1$ we have that $\pi$ is an isomorphism on $\Gamma$, thus it must be injective on $S \subset \Gamma$ too.

Assume now that $\pi_{ |S}$ is injective. We are going to prove that the intersection of every divisor in $|L+M|$ containing $S$ is indeed a smooth rational curve $\Gamma$ satisfying {\em1., 2., 3.}

We will prove it by induction on $d= \dim \F$.

If $d=1$ the statement is trivial.

Assume now $d \geq 2$. The system $|L+M|$ is base point free of projective dimension $n+d$, hence 
the restriction map $H^0(\F, L+M) \to H^0(S, \Oh_{S})$ has nontrivial kernel, since $h^0(S, \Oh_{S})=n+2$.
Thus there exists a divisor $D \in |L+M|$ containing $S$. $D$ is a $(d-1)$-dimensional scheme and we may assume $D$   integral. Indeed,
assuming by a contradiction that  $D$ is not integral then   there should  exist a positive integer $e$ so that we can write 
 $D=D_1+D_2$, with  $D_1 \in |eL|$ and $D_2 \in |M-eL|$ (cf. \cite[Chapter 2]{miles}). But $h^0(X, D_1) \leq e+1$ and $h^0(X, D_2) \leq n-e+1$, i.e.  at most $e$ point in $S$ can be contained in the divisor $D_1$ since $\pi$ is injective on $S$
 and  at most $n-e$ points in $S$ can be contained in the divisor $D_2$ since $S$ imposes independent conditions on $|M|$, which is  absurd.  

Therefore by construction $D$ is  a rational normal scroll of dimension $d-1$ containing $S$, whose projection $\pi : D \to \proj^1$ is the restriction of the projection of $\F$.
Notice that $\deg (M^{d-1}_{|D})=D \cdot M^{d-1}=(L+M) \cdot M^{d-1}=n-d+2$. Thus we may apply the induction argument to $D$, endowed with the line bundles $L_{|D}$, $M_{|D}$. By induction we know that the intersection of every divisor in $H^0(D, L_{|D}+M_{|D})$ containing $S$ is a smooth rational
curve $\Gamma$, and the intersection of every divisor in $H^0(\F, L+M)$  containing $S$ must coincide with it. 

To conclude the statement we need to prove that such smooth rational curve, if it exists, is unique. Assume that $\pi_{ |S}$  is injective. If $\Gamma_0$ is any smooth rational curve satisfying the conditions 1, 2, 3 we have that the restriction map $H^0(\Gamma_0, L+M)=H^0(\Gamma_0, \Oh_{\Gamma_0} (1+n)) \to  H^0(S, \Oh_{S})$ is an isomorphism, since $\Gamma_0 \iso \proj^1$. Thus $H^0(\F, \sI_{\Gamma_0} (L+M))= H^0(\F, \sI_{S} (L+M))$, hence $\Gamma_0$ is contained in the intersection of the sections of $H^0(\F, L+M)$ vanishing on $S$. But we have just proved that this intersection is precisely the smooth rational curve $\Gamma$, thus $\Gamma_0 = \Gamma$.
\end{proof}

\begin{PROP}\label{unicity} Assume $n \geq 2d$  and let $\F$ be a rational normal scroll of dimension $d$ in $\proj^n$. Consider a set  $S \subset \F$ of $n+2$ points in linear general   position.  

Then the family of rational normal curves $\Gamma$ such that $S \subset \Gamma \subset \F$ and  the induced projection $\Gamma \to \proj^1$ has degree $k$, with $1 \leq k \leq d$, is either empty, or it has positive dimension, or it consists of one single curve.
\end{PROP}
\begin{proof}
Let 
$\sG(S)$ be the family of rational normal curves  containing $S$ and 
let 
$\sG^{\vee}(\F,k)$ be 
 the family of rational normal curves $\Gamma$ in $\F$ such that the
induced map $\Gamma \to \proj^1$ has degree $k$. We are going to study the family $\sG^{\vee}(\F,k) \cap  \sG(S).$

If $k=1$ we apply Lemma \ref{k1}. Consider the map $\F(a_1, \ldots, a_d) \to\F$. If there exists a rational normal curve $\Gamma$ with the required properties, then $S$ and $\Gamma$ can be lifted to a set of $n+2$ points and a smooth rational curve in $\F(a_1, \ldots, a_d)$ as described in Lemma \ref{k1}. In particular, it is clear that the family of suitable $\Gamma$'s coincides with the family of liftings of $S$ in $\F(a_1, \ldots, a_d)$. If $S \cap \text{Sing}(\F) = \emptyset$ the lifting is unique and there exists at most one rational normal curve. Otherwise there is a positive dimensional family of liftings.

\hfill\break
Assume now $k \geq 2$. We work by induction on $n - 2d$. Assume at first that $n-2d=0$.

Consider the incidence correspondence
 $$I_k=\{(\Gamma, P_1, \ldots, P_{n+2}): P_i \in \Gamma\} \subset \sG^{\vee}(\F,k) \times \F(a_1, \ldots, a_d)^{n+2}.$$

Proposition \ref{curve_in_scroll} implies that $\dim I_k=(d-1)(n+3-k)+(n+2)$ since the projection $I_k \to  \sG^{\vee}(\F,k) $ is surjective with ($n+2$)-dimensional fibres.

Consider now the map $I_k \to \F(a_1, \ldots, a_d)^{n+2}.$  It is generically injective,  i.e.,
there exists $S \in \F(a_1, \ldots, a_d)^{n+2}$ having precisely one preimage in $I_k$. Indeed, by the polynomial description of the curves we made in Proposition \ref{curve_in_scroll},  and a dimensional computation,  a point in $\F$ imposes $d-1$ independent conditions on  the family of the curves with degree $k$ map to $\proj^1$. 
Similarly,
$n+3-k$ general   points in $\F$ impose precisely $(d-1)(n+3-k)$ conditions on such family, thus  by Proposition \ref{curve_in_scroll} there exists a finite number of suitable rational normal curves through those points.
By fixing one of them and taking  other $k-1$ points on it we obtain  the length $n+2$ scheme $S$ we were looking for.

The map $I_k \to \F(a_1, \ldots, a_d)^{n+2}$, is a quasi-projective morphism  since $I_k \subset \sG \times \F(a_1, \ldots, a_d)^{n+2}$ and $\sG$ is quasi-projective. 

Since there is a point in $\F(a_1, \ldots, a_d)^{n+2}$ with precisely one preimage, by  a Stein factorisation  argument (see \cite[Corollary III.11.5]{Ha}) we see that every fibre is connected. In particular every set of $n+2$ points with zero dimensional fibre (in particular those points avoiding the singular locus of $\F$ by Proposition \ref{number_rnc}) must have precisely one preimage,  proving our  statement.

\hfill\break
Assume now that $n > 2d$.  Fix $S \subset \F$ and assume that there is a finite number of rational normal curve $\Gamma$ such that
 $S \subset \Gamma \subset \F$ and  the induced projection $\Gamma \to \proj^1$ has degree $k$.  With the usual  projection argument used in
 Proposition  \ref{number_rnc} we may assume that $S \cap \Sing(\F)= \emptyset$.

Consider a point $P\in S$ and  the projection $\pi_P$ onto $\proj^{n-1}$.

By  Proposition  \ref{number_rnc}  and induction there is exactly one rational normal curve $\Gamma_0$ in  $\proj^{n-1}$  such that
$ \pi(S\setminus \{P\}) \subset \Gamma_0 \subset \pi_P(\F)$.

But  a rational normal curve in the scroll $\pi_P(\F) \subset\proj^{n-1}$ has one lifting in $\F$ by birationality, thus
it has at most one lifting which passes through $P$, and
we may conclude by induction.
\end{proof}

\hfill\break
Finally we are going to consider families of scrolls which contains a rational normal curve $\Gamma$  passing through $n+2$ fixed points, such that  the induced map $\Gamma \to \proj^1$ has degree $k$.

\begin{PROP}\label{Fk}  Assume that $n$ is an even number and let $d= \frac n2$. Let  $\underline{a}=\{a_1, \ldots, a_d\}$  be such that $\sum a_i=n-d+1$. 

Consider a set $S$  of $n+2$ points in linear general   position in $\proj^n$ and  let $\sF(S,k)_{\underline{a}}$   be the family 
 consisting of  $\F$  isomorphic  to  $\F(\underline{a})$ 
such that   there exists a rational normal curve $\Gamma$  so that $S \subset \Gamma \subset \F$ and the induced map $\Gamma \to \proj^1$ has degree $k$. Then
\begin{itemize}
\item
if $n-k a_i \geq 0$ for every $i$ we have
\begin{eqnarray*}\dim \sF(k)_{\underline{a}}& =&\frac{n^2}{4}+n-2-(k-1)(\frac n2-1)-(\dim \Aut(\underline{a})-\frac{n^2}{4}) \\ & =& \dim \sF(S)_{\underline{a}} - (k-1)(\frac n2-1)\end{eqnarray*}

\item
otherwise $\sF(k)_{\underline{a}}$ is empty.
\end{itemize}
\end{PROP}

\begin{proof} If $n-ka_i < 0$ for some $i$ Proposition \ref{curve_in_scroll} tells us that there are no rational normal curves with the required degree on $\proj^1$, thus $\sF(k)_{\underline{a}}$ is empty. We may therefore assume that this is not the case.

In Thm. \ref{dimensionvarie}  we have seen that $\sF(S)_{\underline{a}}$ is isomorphic to an open subset of $\F(a_1, \ldots, a_d)^{n+2}/\Aut(\F(a_1, \ldots, a_d))$. According to this isomorphism the family we are looking for, $\sF(k)_{\underline{a}}$, coincides with the set of $n+2$ points in $\F(a_1, \ldots, a_d)$ contained in some rational normal curve $\Gamma$, such that $\deg(\Gamma \to \proj^1)=k$ (up to $\Aut(\F(a_1, \ldots, a_d)))$. We have seen how these curve are characterised in Proposition \ref{curve_in_scroll}.

Consider the incidence correspondence
$$I_k=\{(\Gamma, P_1, \ldots, P_{n+2}): P_i \in \Gamma\} \subset \sG^{\vee}(\F(\underline{a}),k)  \times \F(\underline{a})^{n+2}.$$
Thanks to Proposition \ref{curve_in_scroll} we know that $\dim I_k= (d-1)(n+3-k)+n+2$. Consider the projection $\varphi: I_k \to \F(a_1, \ldots, a_d)^{n+2}$. Since the space $\sF(k)_{\underline{a}}$ is $\varphi (I_k) / \Aut(\F(a_1, \ldots, a_d)) $, its dimension is at most $\dim I_k - \dim \Aut(\F(a_1, \ldots, a_d))$:

$$\dim \sF(k)_{\underline{a}} \leq \frac{n^2}{4}+n-2-(k-1)(\frac n2-1)-(\dim \Aut(\underline{a})-\frac{n^2}{4}).$$

To prove the opposite inequality we argue as in the proof of Lemma \ref{unicity}.
Consider the characterisation of rational normal curves  as a set of polynomials of prescribed degree, up to automorphism, given in Proposition \ref{curve_in_scroll}. A simple linear algebra computation shows that  passing through a fixed point imposes at most $d-1=\frac n2 -1$ conditions on those polynomials, thus for $n+3-k$ general   points in $\F(a_1, \ldots, a_d)$ there exists at least one rational normal curve with the right degree. Fixing $n+3-k$ general   points, we can moreover choose the remaining $k-1$ points on this curve. Thus we  find a space of dimension $d\cdot (n+3-k)+(k-1)$ in $\F(a_1, \ldots, a_d)^{n+2}$ which is in the image of the incidence correspondence $I \to \F(a_1, \ldots, a_d)^{n+2}$. Therefore  $\sF(k)_{\underline{a}}$ has dimension at least $d\cdot (n+3-k)+(k-1) - \dim \Aut(\F(a_1, \ldots, a_d))$, or equivalently
$$\dim \sF(k)_{\underline{a}} \geq \frac{n^2}{4}+n-2-(k-1)(\frac n2-1)-(\dim \Aut(\underline{a})-\frac{n^2}{4}).$$
\end{proof}

\begin{PROP}\label{2n-3} Let $S$ be a set of $n+2$ points in linear general position in $\proj^n$ and let  $\sF(S,k)_{\underline{a}}$ and $\sF(S,h)_{\underline{a}}$ be as in Proposition \ref{Fk}. Assume that $k \neq h$ and $k+h \geq \frac n2 +1$. Then
$$\dim (\sF(S,k)_{\underline{a}} \cap \sF(S,h)_{\underline{a}}) \leq 2n-3.$$
\end{PROP}
\begin{proof} By Proposition  \ref{Fk} $\sF(S)_{\underline{a}}$ is  isomorphic to an open subset  
$$ U \subseteq \F(a_1, \ldots, a_d)^{n+2}/\Aut(\F(a_1, \ldots, a_d))$$
and  $\sF(k)_{\underline{a}} $  (respectively $\sF(h)_{\underline{a}} $)   is the image of the incidence correspondence $I_k$ (resp. $I_h$)  modulo $\Aut(\F(a_1, \ldots, a_d))$.

But we can choose an affine covering so that in local coordinates,  both $I_k$ and $I_h$,  are two affine and transverse  subspaces (modulo actions of $\C^{\ast}$).

In particular, since $\dim \sF(k)_{\underline{a}}  + \dim \sF(h)_{\underline{a}} \geq \sF(S)_{\underline{a}}$ by Proposition  \ref{Fk}, we may conclude that
\begin{eqnarray*}
\dim (\sF(k)_{\underline{a}} \cap \sF(h)_{\underline{a}}) &=& \dim \sF(k)_{\underline{a}}  + \dim \sF(h)_{\underline{a}} - \sF(S)_{\underline{a}}\\
& = & 2n-3 - (\dim \Aut(\underline{a})-\frac{n^2}{4})
\end{eqnarray*}
But it is shown in Equation (\ref{d^2}) of Theorem  \ref{dimensionvarie}  that $\dim \Aut(\underline{a})-\frac{n^2}{4}\geq 0$, hence we obtain our inequality
\end{proof}

\section{General   binary curves and scrolls}


A binary curve is a
 stable curve $C =\Gamma_1 \cup \Gamma_2$   where $\Gamma_i\iso \proj^1$ for $i=1,2$,  and the two curves intersect
transversally in $ n+2 $
distinct points. Note that  the arithmetic genus   of $C$ is $p_a(C)=1-\chi(\Oh_C) =n +1$.

As usual by  $K_C$ we denote a canonical divisor.

We denote by  $Bin_{p_a(C)}\subset \overline{\mathcal{M}}_{p_a(C)}$    the locus of binary curves in the moduli space fo genus $p_a(C)$ stable curves.
By a general   binary curve we mean a  curve $C$ such that $[C]$ is general   in $Bin_{p_a(C)}$.

As in the smooth case a first way to stratify binary curves is via the notion of gonality:

\begin{DEF} The \emph{gonality} of $C$, $\operatorname{gon}(C)$, is the smallest number $k$ such that there is a finite morphism $C \to \proj^1$ of degree $k$. A curve $C$ is said to be  \emph{ hyperelliptic} if $\operatorname{gon}(C)=2$.

   \end{DEF}

First note that a  general   binary curve of genus $p_a(C) \geq 3$ is not hyperelliptic. Indeed if a binary curve is hyperelliptic then there is an isomorphism between the two components fixing the $p_a(C)+1$ points, which exists generically if and only if $p_a(C) \leq 2$.  
 
 Conversely,  
 if $C$ is not hyperelliptic
then
	   $K_{C}$ is very ample on $C$
	by  \cite[Theorem 3.6]{CFHR} and
 $\varphi_{|K_{C}|}$ embeds $C$ in $\proj^{n}$ as the union of two
rational normal curves  intersecting in
 $ n+2$ distinct points in linear general   position. Moreover by \cite{FrTe2}
$C$  satisfies Noether's Theorem, which means that the restriction maps
$$H^0(\proj^{n}, \Oh_{\proj^{n}} (k)) \to H^0(C, \omega_C^{\otimes k})$$
are surjective for every $k \geq 0$.
 
 Therefore  we  can 
identify  $C= \Gamma_1\cup\Gamma_2$ and the  points $\Gamma_1\cap \Gamma_2=\{P_0,\cdots , P_{n+1}\}$  with their  canonical embeddings in $\proj^{n}$.    This way,  
  $C$ is a general   binary curve if and only if
$\Gamma_1$ and $\Gamma_2$ are general   rational normal curves passing through  $\{P_0,\cdots , P_{n+1}\}$.

We have the following useful remark.

\begin{REM}\label{proj-curve} Let $C = \Gamma_1\cup \Gamma_2$ be a general   binary curve canonically embedded in $\proj^n$
and let $P_0\in \Gamma_1 \cap\Gamma_2$.

Consider the projection  from $P_0$ onto a general hyperplane and write it as 
$$\pi_{P_{0}} : \proj^n  \dashrightarrow \proj^{n-1}.$$
Then  $C' =\overline{ \pi_{P_{0}}(C)}$  is a general   binary curve of arithmetic genus $p_a(C') =p_a(C)-1$ canonically embedded in $\proj^{n-1}.$

\end{REM}

Now  let us point out that  it is easy to prove that the usual bound on the gonality of smooth curves holds for binary curves as well.

\begin{LEM}\label{rem gonality}  Let $C=\Gamma_1 \cup \Gamma_2 $ be a binary curve of arithmetic genus $p_a  \geq 3$.
 Then $$\operatorname{gon}(C) \leq \lfloor\frac{p_a(C)+3}{2} \rfloor
$$\end{LEM}
\begin{proof}
Let  $\Gamma_1\cap \Gamma_2=\{P_0,\cdots P_{n+1}\}$, where $ n= p_a(C) - 1$.
 The Lemma follows if we show that there exists a morphism $\psi: \Gamma_1 \to \Gamma_2$ of degree $ \lfloor \frac n2 +1 \rfloor$
which is the identity restricted to $\Gamma_1 \cap \Gamma_2$. Indeed,    the morphism
 $\varphi: C \to \proj^1$ defined as $\psi$ on $\Gamma_1$ and the identity on $\Gamma_2$  is well defined and has the required degree.

To show the existence of such $\psi$
let us consider the normalisation $\pi: \hat{C} \iso \Gamma_1 \sqcup \Gamma_2  \to C$  and let  $\{R_i, S_i\} = \pi^{\ast} (P_i)$ where $R_i \in \Gamma_1$ and
$S_i \in \Gamma_2$.

Assume at first that $n$ is even.

 Taking coordinates on $\Gamma_2 \iso \proj^1$
 such a map should be given by a pair of   polynomials of  degree $\frac n2+1$,
$\psi=(q_1, q_2)$ and the conditions  of being the identity on the intersection correspond to ask $\psi(R_i)=S_i$ for  $i=1, \ldots, n+2$.
Now, the vector space of couples $(q_1, q_2)$ with the required degree has dimension $2\cdot (\frac n2 +2)$, while every pair  $(R_i, S_i)$ imposes one single condition on such vector space. In particular there is a 2-dimensional  vector space of $(q_1, q_2)$ with the required properties, i.e., there exists a one dimensional family of  such maps $\psi$.

Similarly, if $n$ is odd we may find a map $\psi: \Gamma_1 \to \Gamma_2$ of degree $\lfloor  \frac n2 +1 \rfloor$ which is the identity restricted to $\Gamma_1 \cap \Gamma_2$.
\end{proof}

\begin{REM}\label{cliff-max}  Let $C=\Gamma_1 \cup \Gamma_2 $ be a binary curve of arithmetic genus $p_a  \geq 3$. 
We can define its Clifford index (as in the smooth case)  as follows
$$\Cliff(C)=\min_{L\in \Pic(C)}\{ \deg(L)- 2 h^0(C,L) +2 \  : \  h^0(C,L)\geq 2, h^1(C,L)\geq 2\}  $$
(see \cite{FrTe1} for a generalisation to singular curves).

In the smooth case it is a classical result that the gonality is maximal if and only if the Clifford index is maximal, i.e.
$$\gon(C)= \lfloor\frac{p_a(C)+3}{2} \rfloor \Longleftrightarrow 
\Cliff(C)= \lfloor\frac{p_a(C)-1}{2} \rfloor$$
We point out that the same holds for binary curves.

Indeed, it is straightforward to see that $\Cliff(C)\leq \gon(C)-2$. Conversely we may apply verbatim  \cite[Theorem 2.3]{cop-mar}
and see that if the Clifford index is not maximal  then there is a one-dimensional family of maps $C\to \proj^1$ of degree 
$ \leq \lfloor  \frac{p_a(C)+3}{2}\rfloor $;  the same  arguments used in   Lemma \ref{rem gonality}   show that there is at least one map
$C\to \proj^1$ of degree $\leq \lfloor  \frac{p_a(C)+1}{2} \rfloor $. 

\end{REM}

To conclude we show that for a general binary curves  there is no scroll of small  dimension containing it. 

\begin{TEO}\label{no_small_scroll} Let $C= \Gamma_1 \cup \Gamma_2$ be  a \textit{general} binary curve in $\proj^n$.

Then there is no rational normal scroll  of minimal degree and of dimension $\leq  \frac n2$ containing $C$.

\end{TEO}

\begin{proof} Let  $C= \Gamma_1 \cup \Gamma_2$ and let $\Gamma_1 \cap \Gamma_2 =\{ P_0, \cdots, P_{n+1}\}$. 
 We argue by induction on $ n$. If $ n=3$ the result is trivial.
 
 Assume that  $\F$ is a scroll of dimension $d $ containing $C$. 
 For $n\geq 4$ we are going to consider the projection from  a point $P_0 \in  \Gamma_1 \cap \Gamma_2$ onto a general hyperplane 
$$\pi_{P_{0}}: \proj^{n} \dashrightarrow  \proj^{n-1}. $$

 We have that $C'=\pi_{P_{0}}(C)$ is again a binary curves (by Remark \ref{proj-curve})   contained 
 in  $ \F'= \pi_{P_0}(\F)$  which is a scroll of dimension $\leq d$. 
 

In particular if $d \leq  \frac {n-1}{2}$ the theorem follows by induction.  

\hfill\break 
Therefore we are left with the case  where $C\subset \F $ a scroll of dimension $d= \frac n2$, where 
$ n$ is an even integer.  

The theorem follows  if we show that under these  hypotheses either   there exists  a scroll $\G \subset \F$  of lower dimension  containing $C$, or 
$ \F $ is singular at one of the points in $\Gamma_1 \cap \Gamma_2$,  say $P_0$. Indeed    in the first case $\dim (\G) < \frac {n-1}{2}$
whereas in the second case 
 taking the projection from the singular point $P_0$   we obtain a scroll 
 $ \F'= \pi_{P_0}(\F) \subset \proj^{n-1}$   of dimension $\leq \frac {n-1}{2}$ containing $C' = \pi_{P_{0}}(C)$.  

So let $C = \Gamma_1\cup\Gamma_2$ be a general binary curve and assume  that  there exists a scroll $ \F(a_1, \ldots, a_d)  \to  \F \subset \proj^{n}$  of degree $\frac n2 +1$ and dimension $\frac n2$ containing $C$.
Let $L$, $M$  be the pullback of $\Oh_{\proj^1}(1)$,  respectively  $\Oh_{\proj^{n}}(1)$.
It is clear that
 $|L|$ induces two maps
$$
 \xymatrix{
\Gamma_1 \ar[r]^{h:1} & \proj^1  & & \Gamma_2  \ar[r]^{k:1} & \proj^1 }. $$
Note that $h \leq \frac n2$ and $k \leq \frac n2$ since $\dim (\F)=\frac n2 $.

\hfill\break 
 If  $h+k < \frac{n+2}{2}$  then by \cite[Chapter 2]{miles}  and Riemann-Roch for $C$  we have
 $$h^0(\F, \Oh_{\F}(L+M) )= \frac{n+2}{2} + n > h^0(C, \Oh_C(L+M) )= h+k + n$$
i.e.,
there exists  $\G \in |L+M|$ containing $C$.  Such a $\G$ is linearly equivalent to a rational normal scroll and it is irreducible
since otherwise every proper component should be degenerate
because of the minimality of  $\deg (\G)$, contradicting $C$ non-degenerate.
Therefore if $h+k < \frac{n+2}{2}$ there exists  a scroll $\G$ of dimension $\leq\frac{n}{2}-1$ containing $C$.

\hfill\break 
If $h+k\geq  \frac{n+2}{2}$  and  $h \neq k$ 
we take the following families
\begin{itemize}
\item ${\mathcal \sB(S)} :=$  the family of binary  curves passing through  a set $S$ consisting of $n+2$ fixed points in linear general position $\{ P_0,\cdots, P_{n+1}\}  $;
\item
$ \sF_{h,k}  =  \sF(S,k)_{\underline{a}} \cap \sF(S,h)_{\underline{a}}:=$
 the family of scrolls of dimension $\frac{n}{2}$ in $\proj^n$  containing a binary curve  passing trough $\{ P_0,\cdots, P_{n+1}\}  $
 with induced maps onto $\proj^1$ of degree $h$, resp. $k$.
\end{itemize}
Consider now  the incidence correspondence
$$
 \xymatrix{
 I=\big\{ (\F,C) \in  \sL_{h,k} \times  {\mathcal C}  \ :  \  \F \supset C \big\}  \ar[r]^{\hspace{ + 25  mm}
 p_2}  \ar[d]^{p_1}  & \sB(S)  \\ 
\sF_{h,k} &  } $$
Recall that  $ \dim \mathcal \sB(S) = 2n-2$ by Theorem \ref{dimensionvarie} and  that $\dim \sF_{h,k}\leq 2n-3$ by  Proposition  \ref{2n-3}.

If  for a general binary curves there  exists a $d$-dimensional scroll containing it then  $p_2$ must be  dominant. In particular 
$\dim I \geq \dim \mathcal \sB(S) = 2n-2$ and moreover  the generic fibre of the map $p_1$ must have positive dimension.

Therefore, fixing $P_0, \cdots, P_{n+1}$ and $\F$,
 our dimension count shows in particular  that there exists a positive dimension family of  rational normal curves in  $\F$ containing $\Gamma_1 \cap \Gamma_2 =\{ P_0, \cdots, P_{n+1}\}$. 
 Therefore by  Proposition  \ref{Fk} and 
Proposition  \ref{number_rnc} we may conclude that one of the points in $\Gamma_1 \cap \Gamma_2$ is contained in $\text{Sing}(\F)$ and we can conclude.   

\hfill\break  
If $h+k\geq  \frac{n+2}{2}$  and  $h=k$ then by Proposition \ref{unicity}  there exists again a   positive dimensional family of rational normal curves passing through the $n+2$ points, 
hence    by Prop.   \ref{number_rnc}  and Prop. \ref{unicity}  $\F$ is singular  at one of the points in $\Gamma_1 \cap \Gamma_2$, and we can conclude. 
 \end{proof}

 \vspace{.5cm}
{\flushleft Marco Franciosi}\\
Dipartimento di Matematica, Universit\`a di Pisa\\
Via Buonarroti 1, I-56127 Pisa (Italy)\\
{\tt marco.franciosi@unipi.it}
\end{document}